\documentclass[11pt]{amsart}

\usepackage{amsmath}
\usepackage{amsfonts}
\usepackage{amssymb}
\usepackage{amscd}
\usepackage{amsthm}
\usepackage{framed}
\usepackage{fullpage}
\usepackage{graphicx}
\usepackage{latexsym}
\usepackage[numbers]{natbib}  

\usepackage{hyperref}
\hypersetup{colorlinks,linkcolor={red},citecolor={blue},urlcolor={blue}}

\newtheorem{theorem}{Theorem}[section]
\newtheorem{lemma}[theorem]{Lemma}
\newtheorem{proposition}[theorem]{Proposition}

\newtheorem{conjecture}[theorem]{Conjecture}

\theoremstyle{definition}
\newtheorem{definition}[theorem]{Definition}

\newtheorem{remark}[theorem]{Remark}

\numberwithin{equation}{section}

\newcommand{\CC}{\mathbb C}
\newcommand{\HH}{\mathbb H}

\newcommand{\NN}{\mathbb N}
\newcommand{\cD}{\mathcal D}
\newcommand{\cA}{\mathcal A}
\newcommand{\cH}{\mathcal H}

\newcommand{\PP}{\mathbb P}
\newcommand{\QQ}{\mathbb Q}
\newcommand{\RR}{\mathbb R}
\newcommand{\ZZ}{\mathbb Z}

\newcommand{\SL}{\mathop{\mathrm {SL}}\nolimits}

\newcommand{\Sp}{\mathop{\mathrm {Sp}}\nolimits}
\newcommand{\Orth}{\mathop{\null\mathrm {O}}\nolimits}

\newcommand{\im}{\mathop{\mathrm {Im}}\nolimits}

\newcommand{\latt}[1]{{\langle{#1}\rangle}}
\newcommand{\ord}{\mathop{\mathrm {ord}}\nolimits}

\def\dim{\operatorname{dim}}

\def\det{\operatorname{det}}
\def\w{\operatorname{w}}

\newenvironment{psmallmatrix}
  {\left(\begin{smallmatrix}}
{\end{smallmatrix}\right)}

\begin{document}

\title[On the vanishing order of Jacobi forms at infinity]{On the vanishing order of Jacobi forms at infinity} 

\author{Jialin Li}

\address{School of Mathematics and Statistics, Wuhan University, Wuhan 430072, Hubei, China}

\email{jlli.math@whu.edu.cn}

\author{Haowu Wang}

\address{School of Mathematics and Statistics, Wuhan University, Wuhan 430072, Hubei, China}

\email{haowu.wangmath@whu.edu.cn}

\subjclass[2020]{11F46, 11F50, 11F55}

\date{\today}

\keywords{Jacobi forms, vanishing order, formal Fourier--Jacobi series, orthogonal modular forms}

\begin{abstract} 
In this paper, we establish two types of upper bounds on the vanishing order of Jacobi forms at infinity. The first type is for classical Jacobi forms, which is optimal in a certain sense. The second type is for Jacobi forms of lattice index. Based on this bound, we obtain a lower bound on the slope of orthogonal modular forms, and we prove that the module of symmetric formal Fourier--Jacobi series on $\Orth(m,2)$ has finite rank. 
\end{abstract}

\maketitle

\section{Introduction}
The theory of classical Jacobi forms was first developed by Eichler and Zagier \cite{EZ85}. Such forms are holomorphic functions of two variables $(\tau,z)\in \HH \times\CC$, which are modular in $\tau$ and quasi-periodic in $z$, and satisfy some boundary conditions at cusps. Later, Gritsenko \cite{Gri88} introduced Jacobi forms of lattice index as Fourier--Jacobi coefficients of modular forms for orthogonal groups $\Orth(m,2)$. In this setting, the single variable $z$ is replaced by multiple variables $\mathfrak{z}\in L\otimes\CC$, where $L$ is an even positive-definite lattice with bilinear form $\latt{-,-}$ and dual lattice $L'$. 

Let $\Gamma$ be a congruence subgroup of $\SL_2(\ZZ)$ and $\chi: \Gamma \to \CC^\times$ be a character of finite order. Let $k$ and $m$ be positive integers. Let $J_{k,L,m}(\Gamma,\chi)$ denote the space of holomorphic Jacobi forms of weight $k$, index $m$ and character $\chi$ on $\Gamma$ for $L$. For any $\varphi \in J_{k,L,m}(\Gamma, \chi)$ with the Fourier expansion
$$
\varphi(\tau,\mathfrak{z})=\sum_{\substack{n\in\QQ, \, \ell \in L'\\ 2nm\geq \latt{\ell,\ell}}} f(n,\ell) q^n \zeta^\ell, \quad q=e^{2\pi i\tau}, \; \zeta^{\ell}=e^{2\pi i\latt{\ell,\mathfrak{z}}},
$$
its vanishing order at infinity is defined as 
$$
\ord_\infty(\varphi)=\min \{ n\in \QQ : \; \text{$f(n,\ell)\neq 0$ for some $\ell \in L'$} \}.
$$
Estimating how large $\ord_\infty(\varphi)$ could be is a natural problem. This type of estimate should have applications to computations of Jacobi forms and orthogonal modular forms. In the literature, there are several known results for classical Jacobi forms, that is, $L=A_1$.

Let $\varphi \in J_{k,A_1,m}(\SL_2(\ZZ))$.  In 1998, Gritsenko and Hulek \cite[Proposition 3.2]{GH98} proved that if $k$ is even, $\varphi$ is a cusp form, and $\ord_\infty(\varphi) > \min\{\frac{k+2m}{12},\, \frac{3k+2m-6}{18}\}$, then $\varphi=0$. For fixed $m$, this bound is of type $O(k)$ and is optimal for sufficiently large $k$.  

It is more difficult to find an optimal upper bound on $\ord_\infty(\varphi)$ for fixed $k$ and large $m$.  In 2022, Aoki \cite{Aok22} proved that if $\ord_\infty(\varphi)>\frac{k+1}{6}\big( \sqrt{2m+1}+1 \big)$, then $\varphi=0$. This is an upper bound of type $O(m^{\frac{1}{2}})$ for fixed $k$ and large $m$. Recently, Gritsenko, Skoruppa and Zagier \cite[Section 5]{GSZ19} proved that if $\varphi$ is a theta block, then $\ord_\infty(\varphi)$ is bounded from above by a constant depending only on $k$. Inspired by these two results, we establish the following upper bound on $\ord_\infty(\varphi)$, which is the best for fixed $k$ and large $m$.  

\begin{theorem}[see Theorem \ref{th:refined}]\label{MTH1}
It $\varphi \in J_{k,A_1,m}(\Gamma,\chi)$ satisfies
\begin{equation}
\ord_\infty(\varphi) \geq \frac{k\left[\SL_2(\ZZ) :\Gamma\right]}{4}\left(2\pi^{\frac{2}{3}}m^{\frac{1}{3}}k^{-\frac{1}{3}}+13\right)\cdot \log\log\big(2\pi^{\frac{2}{3}}m^{\frac{1}{3}}k^{-\frac{1}{3}}+13\big),
\end{equation} 
then it is identically zero. 
\end{theorem}
To prove the theorem, we improve and extend Aoki's argument \cite{Aok22} and employ Robin's inequality related to the Riemann hypothesis \cite{Rob84} to count the number of zeros (with multiplicities) of $\varphi(\tau,z)$ (as a function of $z$) in the fundamental domain $\CC / (\tau\ZZ+\ZZ)$. 

The optimality of the bound in Theorem \ref{MTH1} follows from Proposition \ref{prop:non-trivial}.  For any even $k$, we can find a positive rational number $C$ such that the subspace
$$
J_{k,A_1,m^3}[q^{Cm}]:=\{ \varphi \in J_{k,A_1,m^3}(\SL_2(\ZZ)) : \; \varphi\neq 0, \, \ord_\infty(\varphi) \geq Cm \} 
$$
is not zero for any sufficiently large even $m$ with $Cm\in\ZZ$. Therefore, the best possible upper bound is of type $O(m^{\frac{1}{3}+\varepsilon})$ for fixed $k$ and large $m$, where $\varepsilon$ is an arbitrary positive real number. 

The upper bound on the vanishing order of Jacobi forms at infinity applies to orthogonal modular forms. For example, it can be used to study the slope bounds and modularity of symmetric formal Fourier--Jacobi series. In this context, we need to establish upper bounds on $\ord_\infty(\varphi)$ for large $k$ and $m$ that satisfy $k\leq \epsilon m$ for a constant $\epsilon>0$. The first such bound was recently established by Aoki--Ibukiyama--Poor \cite{AIP24}. They showed that if $\varphi\in J_{k,A_1,dm}(\SL_2(\ZZ))$ satisfies $\ord_\infty(\varphi)\geq m$ and $m>\frac{1}{6}dk$, then $\varphi=0$. We extend their result to the most general case. 

\begin{theorem}[see Theorem \ref{th:k<m}]\label{MTH2}
If $\varphi \in J_{k,L,m}(\Gamma,\chi)$ satisfies $k< 3\cdot 2^{3-2^l}\cdot d_1 m$ and
$$
\ord_\infty(\varphi) > 3^{-2^{-l}} \cdot 2^{1-3\cdot2^{-l}}\cdot\left[\SL_2(\ZZ): \Gamma\right]\cdot\prod_{j=1}^{l}d_j^{2^{j-1-l}}\cdot k^{2^{-l}}m^{1-2^{-l}},
$$
then it is identically zero, where $l$ is the rank of $L$, and $d_j$ are positive integers such that $\bigoplus_{j=1}^l A_1(d_j)$ is a sublattice of $L$ and $d_i\leq d_j$ for any $1\leq i<j\leq l$. Here, $A_1(d)$ denotes the lattice of rank one with the bilinear form $2d x^2$ for $x\in\ZZ$. 
\end{theorem}

We present two applications of the theorem. Let $U$ be an even lattice of signature $(1,1)$ and $M=2U\oplus L$. Let $F$ be a modular form of weight $k$ on the discriminant kernel $\widetilde{\Orth}^+(M)$ (see Section \ref{sec:sFJ} for details). Then $F$ has the Fourier--Jacobi expansion
$$
F(\tau,\mathfrak{z},\omega) = \sum_{m=0}^\infty \phi_m(\tau,\mathfrak{z})\cdot e^{2\pi im\omega}, \quad \phi_m \in J_{k,L,m}(\SL_2(\ZZ)). 
$$
The slope of $F$ is defined by $\varrho(F)=k/\hat{m}$, where $\hat{m}$ is the smallest integer $m$ such that $\phi_m\neq 0$. A lower bound on the slope of Siegel modular forms was given by Eichler \cite{Eic75} fifty years ago (see also \cite[Section 2.5]{BR15}). As an application of Theorem \ref{MTH2}, we establish the first lower bound on the slope of orthogonal modular forms. 

\begin{theorem}[see Theorem \ref{th:slope}]\label{MTH3}
If $F$ is a non-zero modular form of weight $k$ on $\widetilde{\Orth}^+(M)$, then
$$
\varrho(F) \geq 3\cdot 2^{3-2^l}\cdot \prod_{j=1}^l d_j^{-2^{j-1}},
$$    
where $l$ and $d_j$ are positive integers defined in Theorem \ref{MTH2}. 
\end{theorem}

A symmetric formal Fourier--Jacobi series $\Psi$ of weight $k$ on $\widetilde{\Orth}^+(M)$ is a formal series of type $\sum_{m=0}^\infty \phi_m(\tau,\mathfrak{z})\cdot e^{2\pi i m\omega}$ with $\phi_m \in J_{k,L,m}(\SL_2(\ZZ))$ that is invariant under the exchange of $\tau$ and $\omega$. It is widely believed that each $\Psi$ converges and coincides with the Fourier--Jacobi expansion of a modular form on  $\widetilde{\Orth}^+(M)$ (see Conjecture \ref{conj}). As another application of Theorem \ref{MTH2}, we prove the following structure result. 

\begin{theorem}[see Theorem \ref{th:algebraic}]\label{MTH4}
The space $\mathbf{FM}_*(\widetilde{\Orth}^+(M))$ of symmetric formal Fourier--Jacobi series of integral weights has finite rank as a graded module over the ring $\mathbf{M}_*(\widetilde{\Orth}^+(M))$ of modular forms of integral weights. In particular, $\mathbf{FM}_*(\widetilde{\Orth}^+(M))$ is algebraic over $\mathbf{M}_*(\widetilde{\Orth}^+(M))$. 
\end{theorem}

Aoki, Ibukiyama and Poor \cite{AIP24} proved this theorem for $L=A_1(d)$, which plays a crucial role in their proof of the modularity of symmetric formal Fourier--Jacobi series for paramodular groups. Remark that the modularity of symmetric formal Fourier--Jacobi series has been extensively studied in various contexts by Aoki \cite{Aok00}, Ibukiyama--Poor--Yuen \cite{IPY13}, Bruinier--Raum \cite{BR15,BR24}, Xia \cite{Xia22}, Pollack \cite{Pol24}, and Flores \cite{Flo24}. These results have further applications to arithmetic geometry. 

This paper is organized as follows. In Section \ref{sec:preliminary}, we review some basics of Jacobi forms. In Section \ref{sec:classical}, we prove Theorem \ref{MTH1}. In Section \ref{sec:lattice-index}, we prove Theorem \ref{MTH2}. In Section \ref{sec:sFJ}, we review symmetric formal Fourier--Jacobi series on $\Orth(m,2)$ and prove Theorem \ref{MTH3} and Theorem \ref{MTH4}.

\section{Jacobi forms of lattice index}\label{sec:preliminary}
In this section, we give a brief overview of the theory of Jacobi forms of lattice index following \cite{EZ85, Gri88, CG13}. Let $L$ be an even positive definite lattice with bilinear form $\latt{\cdot,\cdot}$ and dual lattice $L'$. Let $k\in\ZZ$ be an integer and $m\in \NN$ be a non-negative integer. The group $\SL_2(\RR)$ acts on the space of holomorphic functions on $\HH\times(L\otimes \CC)$ via
\begin{equation}
\big(\varphi\vert_{k,m} A\big)(\tau,\mathfrak{z}):= (c\tau + d)^{-k} e^{-\pi i m \frac{c\latt{\mathfrak{z},\mathfrak{z}}}{c \tau + d}}\varphi \left( \frac{a\tau +b}{c\tau + d},\frac{\mathfrak{z}}{c\tau + d} \right), \quad A=\begin{psmallmatrix}
a & b \\ c & d
\end{psmallmatrix}\in \SL_2(\RR). 
\end{equation}
The real Heisenberg group associated with $L$ is the set
$$ 
H(L\otimes \RR)=\{ [x,y:r]: x,y\in L\otimes \RR, r \in \RR\}
$$
together with the operation
\begin{equation}\label{eq:Heisenberg}
[x_1,y_1:r_1] \cdot [x_2,y_2:r_2]= \Big[x_1+x_2,y_1+y_2:r_1+r_2+\frac{1}{2}\big( \latt{x_1,y_2}-\latt{x_2,y_1} \big) \Big].
\end{equation}
It acts on the space of holomorphic functions on $\HH\times(L\otimes \CC)$ via
\begin{equation}\label{eq:action-2}
\big(\varphi\vert_{k,m}[x,y:r]\big)(\tau,\mathfrak{z}):= e^{\pi i m ( \latt{x,x}\tau +2\latt{x,\mathfrak{z}}+\latt{x,y}+2r)}  \varphi (\tau, \mathfrak{z}+ x \tau + y).  
\end{equation}
For any $A\in \SL_2(\RR)$ and $h=[x,y:r]\in H(L\otimes \RR)$, we have 
\begin{equation}\label{eq:semidirect}
\Big(\big(\varphi\vert_{k,m}A\big)\vert_{k,m}h\Big)\vert_{k,m}A^{-1} = \varphi\vert_{k,m}(hA^{-1}), \quad \text{where} \; hA^{-1}=[dx - cy,\, ay - bx : r].    
\end{equation}
The integral Heisenberg group is a subgroup of $H(L\otimes\RR)$ defined as
$$
H(L)=\big\{ [x,y:r]: \; x,y\in L, \; r+\latt{x,y}/2\in \ZZ \big\}.
$$
Let $\Gamma$ be a congruence subgroup of $\SL_2(\ZZ)$, and let $\chi: \Gamma \to \CC^\times$ be a character of finite order. Jacobi forms are holomorphic functions on $\HH\times(L\otimes \CC)$ that are invariant under the action of the semidirect product $\Gamma \rtimes H(L)$ and satisfy some boundary conditions.  

\begin{definition}
A holomorphic function $\varphi : \HH \times (L \otimes \CC) \rightarrow \CC$ is called a weak Jacobi form of weight $k$, index $m$ and character $\chi$ on $\Gamma$ for $L$ if it satisfies 
$$
\varphi\vert_{k,m}A=\chi(A) \varphi  \quad \text{and} \quad  \varphi\vert_{k,m}h=\varphi
$$
for any $A\in \Gamma$ and $h\in H(L)$, and if around any cusp represented by $\gamma\in \Gamma \backslash \SL_2(\ZZ)$, it has the Fourier expansion 
\begin{equation}
\big(\varphi\vert_{k,m}\gamma\big)( \tau, \mathfrak{z} )= \sum_{ \substack{ n\in \frac{1}{n_\gamma} + \ZZ, \; \ell \in L' \\ n\geq 0 } }f_\gamma(n,\ell)e^{2\pi i (n\tau + \latt{\ell,\mathfrak{z}})},
\end{equation}
where $n_\gamma$ is a certain integer depending on $\gamma$ and $\chi$. Furthermore, if $f_\gamma(n, \ell)=0$ unless $2nm\geq \latt{\ell,\ell}$ around any cusp $\gamma$, then $\varphi$ is called a holomorphic Jacobi form. Let $J_{k,L,m}^{\w}(\Gamma,\chi)$ and $J_{k,L,m}(\Gamma,\chi)$ denote the $\CC$-vector space of weak Jacobi forms and holomorphic Jacobi forms, respectively. If $\Gamma=\SL_2(\ZZ)$ or $\chi=1$, then we omit $\Gamma$ or $\chi$ for simplification.     
\end{definition}

Remark that weak Jacobi forms of index $0$ are independent of the variable $\mathfrak{z}\in L\otimes\CC$. In fact, they degenerate into the usual holomorphic modular forms on $\Gamma$, that is, $J_{k,L,0}^{\w}(\Gamma, \chi)=M_k(\Gamma, \chi)$. For any positive integer $d$, we denote by $L(d)$ the lattice with the abelian group $L$ and the rescaled bilinear form $d\latt{\cdot,\cdot}$. We note that $J_{k,L,md}(\Gamma,\chi)=J_{k,L(d),m}(\Gamma, \chi)$. Every lattice $L$ of rank one is of type $A_1(d)$, where $A_1$ is the simplest even lattice $\ZZ$ equipped with the bilinear form $2x^2$. The space $J_{k,A_1,m}(\Gamma, \chi)$ is identical to the space $J_{k,m}(\Gamma, \chi)$ of classical Jacobi forms introduced by Eichler--Zagier \cite{EZ85}. 

The structure of $J_{k,m}^{\w}$, that is, the space of weak Jacobi forms with trivial character on $\SL_2(\ZZ)$ for $A_1$, is described in \cite[Theorem 9.4]{EZ85} (see also \cite[Section 2.3]{GW20}). 

\begin{proposition}\label{Prop:weak}
Let $k\in\ZZ$, $m\in\NN$ and $\varphi\in J_{k,m}^{\w}$. If $k$ is even, then $\varphi$ has the unique expression 
$$
\varphi(\tau,z)=\sum_{j=0}^m f_j(\tau)\phi_{-2,1}^j(\tau,z)\phi_{0,1}(\tau,z)^{m-j}
$$
where $f_j\in \mathrm{M}_{k+2j}(\SL_2(\ZZ))$ for any $0\leq j\leq m$, and 
\begin{align*}
\phi_{-2,1}(\tau, z) &= \frac{\vartheta(\tau, z)^2}{\eta(\tau)^{6}} = \zeta + \zeta^{-1} - 2 +O(q) \in J_{-2,1}^{\w}, \\
\phi_{0,1}(\tau, z) & = - \frac{3}{\pi^{2}}\wp(\tau, z)\phi_{-2,1}(\tau, z) = \zeta+\zeta^{-1}+10+O(q) \in J_{0,1}^{\w}. 
\end{align*}
If $k$ is odd, then there exists a unique $\psi\in J_{k+1,m-2}^{\w}$ such that $\varphi=\phi_{-1,2}\psi$, where
$$
\phi_{-1,2}(\tau, z) = \frac{\vartheta(\tau, 2z)}{\eta(\tau)^3} =\zeta-\zeta^{-1} + O(q) \in J_{-1,2}^{\w}. 
$$
Here $q=e^{2\pi i\tau}$ for $\tau\in \HH$, $\zeta=e^{2\pi iz}$ for $z\in\CC$, and 
$$
\eta(\tau)=q^{\frac{1}{24}}\prod_{n=1}^\infty(1-q^n)
$$
is the Dedekind eta function, and 
$$
\vartheta(\tau,z)=\sum_{n\in\ZZ}\left(\frac{-4}{n}\right)q^{\frac{n^2}{8}}\zeta^{\frac{n}{2}} =-q^{\frac{1}{8}}\zeta^{-\frac{1}{2}}\prod_{n=1}^\infty(1-q^{n-1}\zeta)(1-q^n\zeta^{-1})(1-q^n)
$$
is the odd Jacobi theta function, and 
$$
\wp(\tau, z) = \frac{1}{z^{2}} + \sum_{\omega \in (\mathbb{Z} + \tau \mathbb{Z}) \setminus \{0\}} \left( \frac{1}{(z + \omega)^{2}} - \frac{1}{\omega^{2}} \right)
$$
is the Weierstrass elliptic function. 
\end{proposition}

Recall that both $J_{*,L,1}^{\w}=\bigoplus_{k\in\ZZ} J_{k,L,1}^{\w}$ and $J_{*,L,1}=\bigoplus_{k\in\ZZ} J_{k,L,1}$ are free modules of rank $\det(L)$ over $M_*(\SL_2(\ZZ))=\CC[E_4,E_6]$, that is, the ring of usual modular forms of even weight on $\SL_2(\ZZ)$ (see, e.g., \cite[Proposition 2.2]{WW23}).  
For any $L$ of rank $l$, there exist positive integers $d_j$ for $1\leq j\leq l$ such that $\bigoplus_{j=1}^l A_1(d_j)$ is an even sublattice of $L$, and therefore 
$$
J_{k,L,m}^{\w} \subset J_{k, \, \oplus_{j=1}^l A_1(d_j),\, m}^{\w} \quad \text{and} \quad J_{k,L,m} \subset J_{k, \, \oplus_{j=1}^l A_1(d_j),\, m}.
$$
Hence, the space $J_{k, \, \oplus_{j=1}^l A_1(d_j),\, m}^{\w}$ can be determined by Proposition \ref{Prop:weak} and the following isomorphism between $\CC[E_4,E_6]$-modules (see \cite[Theorem 2.4]{WW23})
\begin{equation}\label{eq:tensor}
\begin{split}
J_{*,L_1,1}^{\w} \otimes J_{*,L_2,1}^{\w} &\longrightarrow J_{*,L_1\oplus L_2, 1}^{\w}, \\
(\varphi_1, \varphi_2) &\longmapsto (\varphi_1\otimes\varphi_2)(\tau, \mathfrak{z}_1, \mathfrak{z}_2):=\varphi_1(\tau,\mathfrak{z}_1)\varphi_2(\tau,\mathfrak{z}_2), 
\end{split}
\end{equation}
where $\mathfrak{z}_1\in L_1\otimes\CC$ and $\mathfrak{z}_2\in L_2\otimes\CC$. Consequently, this provides a unified approach to bounding the graded module $J_{*,L,m}^{\w}$. 

We now consider the action of rational elements in the Heisenberg group on Jacobi forms, and describe the modular invariance of the corresponding images. The following result is a generalization of \cite[Section 3]{GW25}, which plays a crucial role in the proof of our main theorems. 

\begin{proposition}\label{Prop:pullback}
Fix $\hat{h}=(\hat{x},\hat{y})$ with $\hat{x}, \hat{y} \in L\otimes\QQ$. For any $\varphi\in J_{k,L,m}^{\w}(\Gamma, \chi)$, we define 
$$
\hat{\varphi}(\tau,\mathfrak{z}):=\big(\varphi\vert_{k,m} [\hat{h}:0]\big)(\tau,\mathfrak{z}).
$$
\begin{enumerate}
\item Recall that $\hat{h}\gamma=(a\hat{x}+c\hat{y}, b\hat{x}+d\hat{y})$ for $\gamma=\begin{psmallmatrix}
a & b \\ c & d     
\end{psmallmatrix}\in \SL_2(\ZZ)$. Then 
$$
\hat{\Gamma}:=\{ \gamma \in \Gamma : \; \hat{h}\gamma - \hat{h} \in L\oplus L \}
$$
is a congruence subgroup of $\SL_2(\ZZ)$. Moreover, for any $\gamma \in \hat{\Gamma}$, we have
$$
\hat{\varphi}\vert_{k,m} \gamma =\chi(\gamma) \cdot e^{\pi im ( \latt{x',y'} + \latt{\hat{x},y'} - \latt{x',\hat{y}} )} \cdot \hat{\varphi},
$$
where $(x',y'):=\hat{h}\gamma - \hat{h}$. 
\item For any $h=[x,y:r]\in H(L)$, we have
$$
\hat{\varphi}\vert_{k,m} h = e^{2\pi im(\latt{\hat{x},y}-\latt{x, \hat{y}})}\cdot \hat{\varphi}. 
$$
\item Around any cusp represented by $\gamma\in\Gamma\backslash\SL_2(\ZZ)$, we consider the Fourier expansions
$$
\varphi\vert_{k,m}\gamma=\sum_{n\in\QQ, \ell\in L'}f_\gamma(n,\ell)e^{2\pi i(n\tau+\latt{\ell, \mathfrak{z}})} \quad \text{and} \quad
\hat{\varphi}\vert_{k,m}\gamma=\sum_{n\in\QQ, \ell\in L'}\hat{f}_\gamma(n,\ell)e^{2\pi i(n\tau+\latt{\ell, \mathfrak{z}})}.
$$
Then we have
$$
\min\{ 2nm-\latt{\ell,\ell} : \; \hat{f}_\gamma(n,\ell)\neq 0 \} \geq \min\{ 2nm-\latt{\ell,\ell} : \; f_\gamma(n,\ell)\neq 0 \}. 
$$
In particular, if $\varphi\in J_{k,L,m}(\Gamma, \chi)$, then $\hat{\varphi}(\tau,0)$ is a holomorphic modular form of weight $k$ and a certain character for $\hat{\Gamma}$. 
\end{enumerate}
\end{proposition}

\begin{proof}
Let $N$ be the smallest positive integer such that both $N\hat{x}$ and $N\hat{y}$ are in $L$. Then $\Gamma(N)\cap \Gamma$ is a subgroup of $\hat{\Gamma}$. It follows that $\hat{\Gamma}$ is a congruence subgroup. For any $\gamma \in \hat{\Gamma}$, by \eqref{eq:semidirect} we verify
\begin{align*}
\varphi\vert_{k,m}[\hat{h}:0]\vert_{k,m} \gamma &= \big(\varphi\vert_{k,m}\gamma\big) \big(\vert_{k,m} \gamma^{-1} \vert_{k,m} [\hat{h}:0]\vert_{k,m} \gamma \big) \\
&= \chi(\gamma) \cdot \varphi\vert_{k,m}[\hat{h}\gamma:0] \\
&= \chi(\gamma) \cdot \varphi\vert_{k,m} [\hat{h}\gamma - \hat{h}: -\latt{x',y'}/2]\vert_{k,m} [0:a_\gamma/2] \vert_{k,m} [\hat{h}:0] \\ 
&= \chi(\gamma) e^{\pi i ma_\gamma} \cdot \varphi\vert_{k,m}[\hat{h}:0],
\end{align*}
where $a_\gamma=\latt{x',y'} + \latt{\hat{x},y'} - \latt{x',\hat{y}}$. For any $h=[x,y:r]\in H(L)$, by \eqref{eq:Heisenberg} we verify 
\begin{align*}
\varphi\vert_{k,m}[\hat{h}:0]\vert_{k,m} h &= \varphi\vert_{k,m} h \vert_{k,m} [0:b_\gamma] \vert_{k,m} [\hat{h}:0] \\
&=e^{2\pi im b_\gamma}\cdot \varphi\vert_{k,m}[\hat{h}:0],
\end{align*}
where $b_\gamma=\latt{\hat{x},y}-\latt{x, \hat{y}}$. Similarly, for any $\gamma\in\SL_2(\ZZ)$, we have 
$$
\varphi\vert_{k,m}[\hat{h}:0]\vert_{k,m}\gamma=\varphi\vert_{k,m}\gamma\vert_{k,m}[\hat{h}\gamma:0]. 
$$
The invariance under the integral Heisenberg group $H(L)$ implies that the Fourier coefficients $f_\gamma(n,\ell)$ depend only on the number $2nm-\latt{\ell,\ell}$ and the class $\ell$ in $L'/mL$. By this fact and \eqref{eq:action-2}, we compute the Fourier expansion of $\varphi\vert_{k,m}\gamma\vert_{k,m}[\hat{h}\gamma:0]$. Comparing this Fourier expansion with that of $\varphi\vert_{k,m}\gamma$, we then deduce the desired inequality. 
\end{proof}

We now define the vanishing order of Jacobi forms at infinity. Let $\varphi: \mathbb{H} \times (L \otimes \mathbb{C}) \to \mathbb{C}$ be a holomorphic function with the expansion 
$$ 
\varphi(\tau,\mathfrak{z})=\sum_{n\in \frac{1}{n_0}+\ZZ, \; n\gg \infty} c_n(\mathfrak{z}) q^n,
$$	
where $n_0$ is a certain positive integer and $n\gg \infty$ means that $n$ is bounded from below. We define the vanishing order of $\varphi$ at infinity as 
\begin{equation}
\mathrm{ord}_\infty(\varphi):=\min\Big\{n\in \frac{1}{n_0}+\ZZ : \; c_n(\mathfrak{z})\neq 0\Big\}.    
\end{equation}
We further define $\mathrm{ord}_\infty(\varphi)=\infty$ if $\varphi=0$. For any holomorphic function $f:\HH \to \CC$,  the vanishing order $\ord_\infty(f)$ is defined similarly. 
In this paper, we focus on the subspace
\begin{equation}
J_{k,L,m}(\Gamma, \chi)[q^N] := \big\{ \varphi \in J_{k,L,m}(\Gamma,\chi): \; \ord_\infty(\varphi) \geq N \big\},    
\end{equation}
where $N\in \QQ$ is a given constant.

\section{Vanishing order of classical Jacobi forms}\label{sec:classical}

In this section, we derive a sharp bound on the vanishing order of classical Jacobi forms at infinity. Our approach is an improvement and extension of Aoki's argument \cite{Aok22}. 

We first review some basic facts about the congruence subgroup. For any positive integer $t$, the congruence subgroup 
$$
\Gamma_1(t)=\big\{\begin{psmallmatrix}
a & b \\ c & d     
\end{psmallmatrix}\in\SL_2(\ZZ): \; a=d=1\bmod t,\; c=0\bmod t\}
$$
has the coset decomposition 
\begin{equation}\label{eq:coset}
\Gamma_1(t) \backslash \SL_2(\ZZ) =\bigsqcup_{\substack{c,d \bmod t\\\mathrm{gcd}(c,d,t)=1}}\Gamma_1(t)\begin{psmallmatrix}
a & b \\ c & d    
\end{psmallmatrix}.
\end{equation}
For $j\in\{1,2\}$, following \cite[Theorem 3.1]{Aok22}, we define 
\begin{equation}
\psi_j(t)=\left\{\begin{array}{ll}
		1, & t=1, \\
		\prod_{\substack{\text{$p$: prime} \\
				p \mid t}}\Big(1-\frac{1}{p^j}\Big), & t\in\ZZ, \; t>1. 
	\end{array}\right.    
\end{equation}
We note that $[\SL_2(\ZZ):\Gamma_1(t)]=t^2\psi_2(t)$. 

We then review the differential operators acting on classical Jacobi forms following \cite[Section 3]{EZ85}. Let $k$ be an even integer, and $\varphi\in J_{k,m}$. We expand $\varphi$ at $z=0$ as
$$
\varphi(\tau,z)=\sum_{\nu=0}^\infty \chi_\nu(\tau)z^\nu.
$$
By \cite[Theorem 3.2]{EZ85}, for any $\nu\geq 0$, the function
\begin{equation}\label{eq:differential}
\big(\xi_\nu\varphi\big)(\tau)=\sum_{0\leq \mu \leq \frac{\nu}{2}} \frac{(-2\pi i m)^\mu (k+\nu-\mu-2)!}{(k+\nu-2) !\mu !} \frac{d^\mu \chi_{\nu-2\mu}(\tau)}{d \tau^\mu}    
\end{equation} 
defines a holomorphic modular form of weight $k+\nu$ and trivial character on $\SL_2(\ZZ)$. In contrast, $\chi_v$ can also be expressed in terms of $\xi_\nu$. In fact, by \cite[Equation (12)]{EZ85}, we have
\begin{equation}\label{eq:converse-differential}
\chi_\nu(\tau)=\sum_{0\leq \mu \leq \frac{\nu}{2}} \frac{(2\pi i m)^\mu (k+\nu-2\mu-1)!}{(k+\nu-\mu-1) !\mu !}\frac{d^\mu \xi_{\nu-2\mu}(\tau)}{d \tau^\mu}.
\end{equation}
We remark that $\xi_\nu\varphi=0$ for any odd $\nu$, because $\varphi$ is an even function with respect to $z$. 

More generally, for any holomorphic function $\phi$ on $\HH\times\CC$, we have the transformation law:
\begin{equation}\label{eq:xi-gamma}
\xi_\nu(\phi\vert_{k,m}\gamma)=\big(\xi_\nu\phi\big) \vert_{k+\nu} \gamma, \quad \gamma\in \SL_2(\ZZ). 
\end{equation}

The specific case $n=0$ of the following lemma recovers \cite[Lemma 5.1]{Aok22},  which is the main tool to establish the upper bound on the vanishing order of Jacobi forms at infinity.

\begin{lemma}\label{lem:zero-mult}
Let $k, m, N$ be positive integers and $\varphi \in J_{k,m}[q^N]$. We assume that $k$ is even. Let $a,b,t\in\ZZ$ with $t\geq 1$ and $\mathrm{gcd}(a,b,t)=1$. We fix $\hat{h}=(\frac{a}{t}, \frac{b}{t})\in\QQ^2$ and $n\in\NN$. If 
\begin{equation}\label{eq:ine-condition}
\sum_{c=0}^{t-1} \psi_1\big(\mathrm{gcd}(t, c)\big) \cdot \max\left\{N-\frac{m c(t-c)}{t^2}, 0\right\}>\frac{t(k+n)}{12} \psi_2(t),
\end{equation}
then the vanishing order of $\big(\varphi\vert_{k,m} [\hat{h}:0]\big)(\tau,z)$ at $z=0$ is larger than $n$. 
\end{lemma}

\begin{proof}
We first prove that $\xi_\nu(\varphi\vert_{k,m}[\hat{h}:0])=0$ for any $0\leq \nu\leq n$ under the condition \eqref{eq:ine-condition}. There exists $g\in \SL_2(\ZZ)$ such that $(0,1)g=(a,b)\,\bmod t\ZZ^2$. By \eqref{eq:xi-gamma} and \eqref{eq:semidirect}, we have  
\begin{equation}\label{eq:reduction}
\xi_\nu\big(\varphi\vert_{k,m}[\hat{h}:0]\big)\vert_{k+\nu}\gamma = \xi_\nu\big(\varphi\vert_{k,m}[\hat{h}:0]\vert_{k,m}\gamma\big)=\xi_\nu\big(\varphi\vert_{k,m}[\hat{h}\gamma:0]\big), \quad \gamma \in \SL_2(\ZZ).    
\end{equation}
Therefore, we can assume $a=0$ and $b=1$ without loss of generality.  From Proposition \ref{Prop:pullback} and \eqref{eq:xi-gamma}, we conclude that $\xi_\nu\big(\varphi\vert_{k,m}[\hat{h}:0])$ is a holomorphic modular form of weight $k+\nu$ and a certain character for $\hat{\Gamma}=\{ \gamma \in \SL_2(\ZZ): \; \hat{h}\gamma=\hat{h} \bmod \ZZ^2\}$. We find $\hat{\Gamma}=\Gamma_1(t)$ for $\hat{h}=(0,1/t)$. Next, we estimate the vanishing order of $\xi_\nu\big(\varphi\vert_{k,m}[\hat{h}:0])$ at infinity by two distinct methods.
		
On the one hand, we apply the valence formula to the $\SL_2(\ZZ)$-modular form
$$
\prod_{\gamma\in \Gamma_1(t)\backslash \SL_2(\ZZ)}\Big(\xi_\nu\big(\varphi\vert_{k,m}[\hat{h}:0]\big)\Big)\vert_{k+\nu}\gamma,
$$
and find that the vanishing order of $\xi_\nu\big(\varphi\vert_{k,m}[\hat{h}:0])$ at infinity can not be too large. More precisely, for nonzero $\xi_\nu\big(\varphi\vert_{k,m}[\hat{h}:0])$, we derive
\begin{equation}\label{eq:upper-bound}
\sum_{\gamma \in \Gamma_1(t) \backslash \SL_2(\ZZ)} \mathrm{ord}_\infty\big(\xi_\nu(\varphi\vert_{k,m}[\hat{h}:0])\vert_{k+\nu}\gamma\big) \leq \frac{k+\nu}{12}\left[ \SL_2(\ZZ) : \Gamma_1(t) \right]=\frac{ t^2(k+\nu)}{12} \psi_2(t).
\end{equation}
		
On the other hand, a lower bound of the order can be derived from the expression of $\varphi$ in terms of generators. Note that $\varphi/\Delta^N\in J_{k-12N,m}^{\w}$. By Proposition \ref{Prop:weak}, $\varphi$ has the expression
$$
\varphi(\tau,z)=\Delta(\tau)^N \sum_{j=0}^{m} f_j(\tau) \phi_{-2,1}(\tau,z)^j \phi_{0,1}(\tau,z)^{m-j}, 
$$
where $f_j\in M_{k-12N+2j}(\SL_2(\ZZ))$. 

Following \cite[Section 3]{GSZ19}, for any $x\in\RR$ and $\phi=\sum_{n\geq 0, r\in\ZZ} c(n,r)q^n\zeta^r \in J_{k,m}^{\w}$, we define 
$$
\ord(\phi, x) := \min\{ n+rx+mx^2 : \; \text{$n$, $r$ such that $c(n,r)\neq 0$} \}.
$$
Let $x,y\in\QQ$. By definition, we obtain 
$$
\ord_\infty(\phi\vert_{k,m}[x,y:0]) = \ord(\phi, x).
$$
It is not difficult to calculate $\ord(\phi,x)$, because the  coefficients $c(n,r)$ depend only on the number $4nm-r^2$ and the class of $r$ in $\ZZ / 2m\ZZ$. When $m=1$ and $k=-2$ or $0$, if $4n-r^2<0$ and $c(n,r)\neq 0$, then there exists an integer $\lambda$ such that $r=2\lambda+1$ and $n=\lambda(\lambda+1)$. Moreover, $c(\lambda(\lambda+1),2\lambda+1)=c(0,1)\neq 0$. It follows that 
\begin{align*}
\ord(\phi, x)&=\min\{ \lambda(\lambda+1) + (2\lambda+1)x + x^2 : \lambda\in\ZZ  \}\\
&=\min\{ (x+\lambda)(x+\lambda+1) : \lambda\in\ZZ \}   \\
&=-(x-[x])(1-x+[x]),
\end{align*}
where $[x]$ denotes the integer part of $x$. 

By the discussions above and Proposition \ref{Prop:pullback} (3), for general $(x,y)\in\QQ^2$, we derive
$$
\mathrm{ord}_\infty(\varphi\vert_{k,m}[x,y:0]) = \ord(\varphi, x) \geq\max\big\{ N-m(x-[x])(1-x+[x]), \; 0\big\}. 
$$
Since the derivative does not reduce the vanishing order at infinity, we obtain
$$
\mathrm{ord}_\infty\big(\xi_\nu(\varphi\vert_{k,m}[x,y:0])\big) \geq\max\big\{ N-m(x-[x])(1-x+[x]), \; 0\big\}. 
$$
We then deduce from \eqref{eq:reduction} the lower bound
\begin{equation}\label{eq:lower-bound}
\sum_{\gamma \in \Gamma_1(t) \backslash \SL_2(\ZZ)} \mathrm{ord}_\infty\big(\xi_\nu(\varphi\vert_{k,m}[\hat{h}:0])\vert_{k+\nu}\gamma\big) \geq \sum_{c=0}^{t-1} t\psi_1\big(\mathrm{gcd}(t,c)\big) \cdot \max\left\{N-\frac{m c(t-c)}{t^2}, \; 0\right\},
\end{equation}
here we use the coset decomposition \eqref{eq:coset} and the identity
$$
\#\{ d\in\ZZ : 0\leq d<t,\; \mathrm{gcd}(c,d,t)=1 \} = t\psi_1\big( \mathrm{gcd}(t,c) \big).
$$
Combining \eqref{eq:ine-condition}, \eqref{eq:upper-bound} and \eqref{eq:lower-bound} together, we derive that $\xi_\nu\big(\varphi\vert_{k,m}[\hat{h}:0]\big)=0$ for $0\leq \nu\leq n$. We therefore prove the desired lemma by an analogue of \eqref{eq:converse-differential}. 
\end{proof} 

We derive a simplified version of Lemma \ref{lem:zero-mult} focusing on the left-hand side of \eqref{eq:ine-condition} to $c=0$.
	
\begin{lemma}\label{lem:zero-mult-0}
Let $k, m, N$ be positive integers and $\varphi \in J_{k,m}[q^N]$. We assume that $k$ is even. Let $a,b,t\in\ZZ$ with $t\geq 1$ and $\mathrm{gcd}(a,b,t)=1$. We fix $\hat{h}=(\frac{a}{t}, \frac{b}{t})\in\QQ^2$ and $n\in\NN$. If 
\begin{equation}\label{eq:ine-condition-0}
N>\frac{t(k+n)\psi_2(t)}{12\psi_1(t)}, 
\end{equation}
then the vanishing order of $\big(\varphi\vert_{k,m} [\hat{h}:0]\big)(\tau,z)$ at $z=0$ is larger than $n$. 
\end{lemma}

We use the Robin inequality below to estimate $\psi_2(t)/\psi_1(t)$. 

\begin{lemma}[Théorème 2 in \cite{Rob84}]\label{lem:Robin}
Let $x\geq 3$ be an integer and let $\sigma(x)$ be the sum of positive divisors of $x$. Then
$$ 
e^{\gamma}\log\log(x)+\frac{0.6483}{\log\log(x)}\geq \frac{\sigma(x)}{x}\geq \frac{\psi_2(x)}{\psi_1(x)},
$$
where $\gamma\approx 0.5772$ is the Euler--Mascheroni constant.
\end{lemma} 

Remark that Robin proved \cite[Théorème 1]{Rob84} that the Riemann hypothesis is true if and only if 
$$
e^{\gamma}\log\log(x)\geq \frac{\sigma(x)}{x}, \quad \text{for any $x>5040$}.
$$
	
We now state the main result of this section. 
	
\begin{theorem}\label{th:refined}
Let $k$ and $m$ be positive integers. Let $\Gamma$ be a congruence subgroup of $\SL_2(\ZZ)$ and $\chi$ be a character of $\Gamma$ of finite order. For any positive integer $N$ no less than 
\begin{equation}
\frac{k\left[\SL_2(\ZZ) :\Gamma\right]}{4}\left(2\pi^{\frac{2}{3}}m^{\frac{1}{3}}k^{-\frac{1}{3}}+13\right)\cdot \log\log\big(2\pi^{\frac{2}{3}}m^{\frac{1}{3}}k^{-\frac{1}{3}}+13\big),
\end{equation} 
we have $J_{k,m}(\Gamma,\chi)[q^N]=0$.
\end{theorem}
\begin{proof}
We prove a more general version of this theorem. 
Let $C_1>0$ and $C_2\geq 1$ be real numbers such that for any $t\geq 0$,
$$
\frac{(A+C_2)^2(A+C_2-2)}{8} - \frac{(A+C_2)(A+C_2+1)(A+C_2+2)}{12}\geq C_1 t^3.
$$
Comparing the coefficient of $t^3$ on both sides of this inequality, it follows that $C_1\leq \frac{1}{24}$. 

We claim that if $N$ is no less than	
$$
\frac{k[\SL_2(\ZZ):\Gamma]}{12}\left(\Big(\frac{\pi^2m}{3C_1k}\Big)^{\frac{1}{3}}+C_2\right)\left(e^{\gamma}\log\log\left(\Big(\frac{\pi^2m}{3C_1k}\Big)^{\frac{1}{3}}+C_2\right)+\frac{0.6483}{\log\log\Big(\big(\frac{\pi^2m}{3C_1k}\big)^{\frac{1}{3}}+C_2\Big)}\right), 
$$
then $J_{k,m}[q^N]=0$. Recall that $\gamma\approx 0.5772$ is the Euler--Mascheroni constant.

\vspace{3mm}
		
We first consider the specific case of $k\in 2\ZZ$, $\Gamma=\SL_2(\ZZ)$, and $\chi=1$.
Let $A=\big(\frac{\pi^2m}{3C_1k}\big)^{\frac{1}{3}}$. We want to determine possible integers $n=n_t$ such that \eqref{eq:ine-condition-0} holds for $1\leq t\leq [A+C_2]$, where $[x]$ denotes the integer part of $x\in\RR$ as before. Let us define 
$$
f(x)=e^{\gamma}\log\log(x)+\frac{0.6483}{\log\log(x)}, \quad x>1. 
$$
We assert that for any $1\leq t \leq [A+C_2]$,
$$
f(A+C_2)=\left(e^{\gamma}\log\log\big(A+C_2\big)+\frac{0.6483}{\log\log(A+C_2)}\right)\geq \frac{\psi_2(t)}{\psi_1(t)}.
$$
Firstly, $f(x)$ increases monotonically for $x\geq 4$. When $[A+C_2]\geq 4$, we deduce from Lemma \ref{lem:Robin} that for any $4\leq t \leq  [A+C_2]$,
$$
f\left(A+C_2\right)\geq f(t)\geq \frac{\psi_2(t)}{\psi_1(t)}. 
$$
Secondly, for $t=1,2,3$, the average value inequality yields 
$$ 
\min\{f(x):\; x>1\}>2\sqrt{e^\gamma\times 0.64}>2\geq 1+t^{-1}\geq\frac{\psi_2(t)}{\psi_1(t)}. 
$$ 
This completes the proof of the assertion. 
		
Let $\varphi\in J_{k,m}[q^N]$.  For any $1\leq t\leq [A+C_2]$, by assumption, we have 
$$
N\geq \frac{k(A+C_2)f(A+C_2)}{12}\geq \frac{k(A+C_2)\psi_2(t)}{12\psi_1(t)}> \frac{t\left(k+ k\big[\frac{A+C_2}{t}\big]-k-1\right)\psi_2(t)}{12\psi_1(t)}.
$$
By Lemma \ref{lem:zero-mult-0}, for any $1\leq t\leq [A+C_2]$ and integers $a$, $b$ with $0\leq a, b <t$ and $\mathrm{gcd}(a,b,t)=1$, the vanishing order of $\big(\varphi\vert_{k,m} [\frac{a}{t},\frac{b}{t}:0]\big)(\tau,z)$ at $z=0$ is at least $k[\frac{A+C_2}{t}]-k$. It follows that the vanishing order of $\varphi(\tau,z)$ at $z=\frac{a}{t}\tau+\frac{b}{t}$ is at least $k[\frac{A+C_2}{t}]-k$. 

The points of type $\frac{a}{t}\tau+\frac{b}{t}$ are distinct modulo $\tau\ZZ+\ZZ$, and the number of such points is given by $[\SL_2(\ZZ): \Gamma_1(t)]=t^2\psi_2(t)$. From the famous identity $\prod_{p:\, \text{prime}}(1-\frac{1}{p^2})=\frac{6}{\pi^2}$, we derive $\psi_2(t)>\frac{6}{\pi^2}$ for any $t\geq 1$. Therefore, if $\varphi$ is not identically zero, then the number of zeros (with multiplicity) of $\varphi(\tau,z)$ (as a function of $z$) in the fundamental domain $\CC / (\tau\ZZ+\ZZ)$ is at least
\begin{align*}
&\sum_{t=1}^{[A+C_2]} t^2\psi_2(t) \cdot \left( k\Big[\frac{A+C_2}{t}\Big]-k\right) \\
> & \frac{6}{\pi^2}\sum_{t=1}^{[A+C_2]} t^2 \cdot \left( k\Big[\frac{A+C_2}{t}\Big]-k\right) \\ 
> & \frac{6k}{\pi^2}\sum_{t=1}^{[\frac{1}{2}(A+C_2)]} t^2 \cdot \left( \frac{A+C_2}{t}-2\right) \\
= & \frac{6k}{\pi^2}\left((A+C_2)\sum_{t=1}^{[\frac{1}{2}(A+C_2)]} t \; - \; 2\sum_{t=1}^{[\frac{1}{2}(A+C_2)]} t^2\right)\\
> & \frac{6k}{\pi^2} \left( \frac{(A+C_2)^2(A+C_2-2)}{8} - \frac{(A+C_2)(A+C_2+1)(A+C_2+2)}{12} \right) \\
\geq & \frac{6kC_1A^3}{\pi^2}=2m.
\end{align*}
We recall from \cite[Theorem 1.2]{EZ85} that the number of zeros (with multiplicity) of any non-zero $\phi\in J_{k,m}$ in the fundamental domain $\CC / (\tau\ZZ+\ZZ)$ is exactly $2m$. Therefore, $\varphi$ is identically zero. This proves our claim in the particular case of $k\in 2\ZZ$, $\Gamma=\SL_2(\ZZ)$, and $\chi=1$.
		
We now consider the general case. Let $\varphi\in J_{k,m}(\Gamma,\chi)[q^N]$. Assume that the order of $\chi$ is $d$. We set $R=2d\cdot\left[\SL_2(\ZZ) :\Gamma\right]$ and find 
$$
\Phi:=\prod_{\gamma\in \Gamma\backslash \SL_2(\ZZ)} \big(\varphi\vert_{k,m}\gamma\big)^{2d}\in J_{Rk,Rm}[q^{2dN}].
$$ 
Then the particular case of the claim implies that $\Phi$ is identically zero, and therefore $\varphi$ is also identically zero. We thus prove the claim in the general case.  

\vspace{3mm}
		
Taking $(C_1,C_2)=(\frac{1}{24},13)$, we find that 
\begin{align*}
& \frac{(t+C_2)^2(t+C_2-2)}{8}-\frac{(t+C_2)(t+C_2+1)(t+C_2+2)}{12}- C_1t^3 \\ 
= & \frac{27t^2+191t+117}{24}>0    
\end{align*}
for any $t>0$. Since $A+C_2\geq 13$ in this case and the inequality 
$$
e^{\gamma}\log\log(x)+\frac{0.6483}{\log\log(x)}\leq 3\log\log(x) 
$$ 
holds for any $x\geq 13$, we prove the desired result.
\end{proof}

We now establish a bound on $N$ such that $J_{k,m}[q^N]\neq 0$ for large $m$. 

\begin{proposition}\label{prop:non-trivial}
Let $k>9$ be an even integer and $C$ be a positive rational number such that $k>9+16C^\frac{3}{2}$. Then $J_{k,m^3}[q^{Cm}]\neq 0$ for any sufficiently large even $m$ that satisfies $Cm\in\ZZ$. 
\end{proposition}

\begin{proof}
Let $\varphi=\sum_{n\in\NN,r\in\ZZ}c(n,r)q^n\zeta^r\in J_{k,m^3}^{\w}[q^{Cm}]$. Note that $c(n,r)=c(n,-r)$. Then $\varphi\in J_{k,m^3}$ if and only if $c(n,r)=0$ for all $n$ and $r$ that satisfy $4nm^3<r^2$ and $0\leq r\leq m^3$. Therefore,
$$
\dim J_{k,m^3}[q^{Cm}] \geq \dim J^{\w}_{k,m^3}[q^{Cm}]-R,
$$
where 
$$
R=\sharp\left\{(n,r)\in\ZZ^2:\; 4nm^3<r^2,\; Cm\leq n< m^3/4,\; 0\leq r\leq m^3\right\}. 
$$
The equality $J^{\w}_{k,m^3}[q^{Cm}]= \Delta^{Cm} \cdot J^{\w}_{k-12Cm,m^3}$ implies
$$
\dim J^{\w}_{k,m^3}[q^{Cm}]=\dim J^{\w}_{k-12Cm,m^3},
$$ 
and therefore we deduce from Proposition \ref{Prop:weak} that 
$$
\dim J^{\w}_{k,m^3}[q^{Cm}]=\sum_{j=0}^{m^3}\dim M_{k-12Cm+2j}(\SL_2(\ZZ)).
$$
For any positive integer $n$, we have
\[
\sum_{j=0}^n \dim M_{2j}(\SL_2(\ZZ)) =
\begin{cases}
\frac{(n + 3)^2}{12}, & \text{if } n \equiv 3 \pmod{6}, \\
\frac{(n + 1)(n + 5)}{12}, & \text{if } n \equiv 1 \text{ or } 5 \pmod{6}, \\
\frac{(n + 2)(n + 4)}{12}, & \text{if } n \equiv 2 \text{ or } 4 \pmod{6}, \\
\frac{n^2 + 6n + 12}{12}, & \text{if } n \equiv 0 \pmod{6}.
\end{cases}
\]
It follows that
\begin{align*}
\dim J^{\w}_{k-12Cm,m^3} \geq& \frac{(k-12Cm+2m^3)^2+12(k-12Cm+2m^3)+20}{48} \\  
=& \frac{1}{12}m^6 - Cm^4  + \left(\frac{k}{12} + \frac{1}{2}\right)m^3\\ 
&+ 3C^2m^2 - \left(\frac{1}{2}Ck + 3C\right)m + \frac{k^2 + 12k + 20}{48}.  
\end{align*}
We also estimate
\begin{align*}
-R&=-\left(\frac{m^3}{4}-Cm\right)(m^3+1)+\sharp\left\{ (n,r)\in\ZZ^2:\,  4nm^3\geq r^2,\, Cm\leq n< m^3/4,\, 0\leq r\leq m^3 \right\} \\ &>-\frac{m^6}{4}+Cm^4-\frac{m^3}{4}+Cm+\sum_{t=Cm}^{m^3/4-1}2m^{3/2}\sqrt{t}\\
&=-\frac{m^6}{4}+Cm^4-\frac{m^3}{4}+Cm+\sum_{t=Cm+1}^{m^3/4}2m^{3/2}\sqrt{t}-m^3+2C^{\frac{1}{2}}m^2\\
&>-\frac{m^6}{4}+Cm^4-\frac{5m^3}{4}+ 2C^{\frac{1}{2}}m^2+Cm+2m^{3/2}\int_{Cm}^{m^3/4}\sqrt{t}dt \\
&=-\frac{m^6}{12}+Cm^4+\left(-\frac{5}{4}-\frac{4}{3}C^{3/2}\right)m^3+2C^{\frac{1}{2}}m^2+Cm. \end{align*}
Based on these estimations, we derive
$$ 
\dim J_{k,m^3}[q^{Cm}] \geq \frac{k-9-16C^{\frac{3}{2}}}{12}m^3+ \left(3C^2+2C^{\frac{1}{2}}\right)m^2 - \left(\frac{1}{2}Ck + 2C\right)m + \frac{k^2 + 12k + 20}{48}. 
$$
Therefore, $J_{k,m^3}[q^{Cm}]\neq 0$ for sufficiently large even $m$ under our assumptions. 
\end{proof}

As an illustrative example, taking $k=10$ and $C=1/8$, we derive from the above theorem that $J_{10,8^3m^3}[q^{m}]\neq 0$ for any positive integer $m$.
	
Remark that Theorem \ref{th:refined} actually shows that for any fixed $k$ and $\alpha>1/3$, there exists a positive constant $C$ such that $J_{k,m}[q^{Cm^\alpha}]=0$ for any positive integer $m$. However, by Proposition \ref{prop:non-trivial}, such $\alpha$ cannot be replaced by any number less than or equal to $1/3$. Therefore, the bound in Theorem \ref{th:refined} is optimal in a sense. 

Theorem \ref{th:refined} can be generalized to high-rank lattices. In principle, we obtain an upper bound of type $O(m^{\alpha_L})$ on $N$ such that $J_{k,L,m}[q^N]\neq 0$, and the exponent $\alpha_L$ approaches $1$ as the rank of $L$ approaches infinity. However, it seems difficult to determine the optimal exponent $\alpha_L$, and such an upper bound is not very meaningful compared to the upper bound established in the next section.

\section{Vanishing order of Jacobi forms of lattice index}\label{sec:lattice-index}
In this section, we establish an upper bound of the vanishing order of Jacobi forms at infinity under the assumption $k\leq \varepsilon m$ for some constant $\varepsilon$ depending only on the lattice $L$. 

We first generalize Lemma \ref{lem:zero-mult} to the case of higher rank, which describes a relation between the upper bounds for rank $l$ and for rank $l-1$. 
	
\begin{lemma}\label{lem:induction}
Let $L$ be an even positive definite lattice. Let $k$, $d$, $t$, $N$ be positive integers. We set $R:=[\SL_2(\ZZ):\Gamma_1(t)]=t^2\psi_2(t)$. Suppose that there is a function $f(k,L)$ such that if $N>f(k,L)$ then $J_{k,L,1}[q^N]=0$. We assume that $\varphi\in J_{k,L\oplus A_1(d),1}[q^N]$ has the expression
\begin{equation}\label{eq:even-expression_phi}
\varphi =\Delta(\tau)^N \sum_{j=0}^{d} f_j(\tau,\mathfrak{z}) \phi_{-2,1}(\tau,z)^j \phi_{0,1}(\tau,z)^{d-j},
\end{equation}
where $\mathfrak{z}\in L\otimes\CC$, $z\in A_1(d)\otimes\CC$, and $f_j\in J_{k+2j,L,1}^{\w}$ for $0\leq j \leq d$. Let $\hat{h}=\big(\frac{a}{t},\frac{b}{t}\big)\in \big(A_1(d)\otimes\QQ\big)^2$, where $a$ and $b$ are integers that satisfy $\mathrm{gcd}(a,b,t)=1$. If 
$$
t^2\sum_{c=0}^{t-1}\psi_1\big(\mathrm{gcd}(t, c)\big) \cdot \max\left\{N-\frac{d c(t-c)}{t^2},\; 0\right\}>f\big(Rtk,L(Rt)\big),
$$ 
then the pullback $\big(\varphi\vert_{k,1}[\hat{h}:0]\big)(\tau,\mathfrak{z},0)$ is identically zero. 
\end{lemma}

\begin{proof}
For any $\gamma\in \SL_2(\ZZ)$, we have 
$$
\varphi\vert_{k,1}[\hat{h}:0]\vert_{k,1}\gamma = \varphi\vert_{k,1}[\hat{h}\gamma :0]. 
$$
Similarly to the proof of Lemma \ref{lem:zero-mult}, we can assume that $a=0$ and $b=1$ without loss of generality. In this case, we derive from Proposition \ref{Prop:pullback} that 
$$
\big(\varphi\vert_{k,1}[\hat{h}:0]\big)(\tau,\mathfrak{z},0) \in  J_{k,L,1}(\Gamma_1(t),\chi),
$$
where $\chi$ is of character of $\Gamma_1(t)$ and $\chi^t=1$. Moreover, 
$$
\Psi(\tau,\mathfrak{z}):=\left(\prod_{\gamma\in \Gamma_1(t)\backslash \SL_2(\ZZ)}\big(\varphi\vert_{k,1}[\hat{h}:0]\big)\vert_{k,1} \gamma \right)^t(\tau,\mathfrak{z},0) \in  J_{Rtk,L(Rt),1}.
$$
Similarly to the proof of Lemma \ref{lem:zero-mult}, the vanishing order of $\Psi$ at infinity is no less than
$$
t\sum_{c=0}^{t-1} t\psi_1\big(\mathrm{gcd}(t, c)\big) \cdot \max\left\{N-\frac{d c(t-c)}{t^2},\; 0\right\}>f\big(Rtk, L(Rt)\big).
$$ 
By assumption, it follows that $\Psi=0$, and therefore $\big(\varphi\vert_{k,1}[\hat{h}:0]\big)(\tau,\mathfrak{z},0)=0$. 
\end{proof}

We then formulate an upper bound of the vanishing order of Jacobi forms at infinity. 

\begin{theorem}\label{th:k<m}
Let $k$, $m$, $n$, $N$ be positive integers. Let $L$ be an even positive definite lattice of rank $n$ that satisfies $\bigoplus_{j=1}^n A_1(d_j)<L$, where $d_1,\dots,d_n$ are positive integers and $d_i\leq d_j$ for any $1\leq i<j\leq n$. Let $\Gamma<\SL_2(\ZZ)$ be a congruence subgroup, and $\chi:\Gamma\to \CC^\times$ be a character of finite order. We set $C_n=3^{-2^{-n}} \cdot 2^{1-3\cdot2^{-n}}$. If $k< 3\cdot 2^{3-2^n}\cdot d_1 m$ and 
$$
N>C_n\cdot\left[\SL_2(\ZZ): \Gamma\right]\cdot\prod_{j=1}^{n}d_j^{2^{j-1-n}}\cdot k^{2^{-n}}m^{1-2^{-n}},
$$
then $J_{k,L,m}(\Gamma,\chi)[q^N]=0$.
\end{theorem}

\begin{proof}
It suffices to prove the assertion that if $\varphi\in J_{k,\oplus_{j=1}^n A_1(d_j),1}(\Gamma,\chi)[q^N]$ satisfies 
$$
k<3\cdot 2^{3-2^n}\cdot d_1 \quad \text{and} \quad N > C_n\cdot\left[\SL_2(\ZZ): \Gamma\right]\cdot k^{2^{-n}}\cdot \prod_{j=1}^{n}d_j^{2^{j-1-n}},
$$ 
then $\varphi$ is identically zero, because $\prod_{j=1}^{n}m^{2^{j-1-n}}=m^{1-2^{-n}}$. 

\vspace{3mm}
		
We first consider the case of $n=1$, $\Gamma=\SL_2(\ZZ)$, and $\chi=1$. In this case, $C_1=1/\sqrt{6}$ and $k<6d_1$. Let us first assume that $k$ is even. Suppose that $N>\sqrt{kd_1/6}$ and $\varphi\in J_{k,d_1}[q^{N}]$. We choose a sufficiently large prime $t$ that satisfies 
$$
t>\sqrt{2d_1+1}, \quad N>\sqrt{kd_1/6}\cdot\sqrt{1+t^{-1}}, \quad 1+t^{-1}<6d_1/k.
$$
We verify
\begin{align*}
&\sum_{c=0}^{t-1}\max\left\{N-\frac{d_1c(t-c)}{t^2},0\right\} \\
>&\sum_{c=0}^{t-1}\max\left\{\sqrt{\frac{kd_1(1+t^{-1})}{6}}-\frac{d_1ct}{t^2},0\right\}\\ 
= &\frac{d_1}{t}\sum_{c=0}^{t-1}\max\left\{t\sqrt{\frac{k(1+t^{-1})}{6d_1}}-c,0\right\}\\
\geq & \frac{d_1}{t}\times\sum_{c=0}^{\big[t\sqrt{\frac{k(1+t^{-1})}{6d_1}}\big]}\left(t\sqrt{\frac{k(1+t^{-1})}{6d_1}}-c\right)\\
\geq & \frac{d_1}{t}\times \left(t\sqrt{\frac{k(1+t^{-1})}{6d_1}}\right)^2\times \frac{1}{2}\\
=&\frac{kt\psi_2(t)}{12\psi_1(t)}. 
\end{align*}
By Lemma \ref{lem:zero-mult}, $\big(\varphi\vert_{k,1}[\frac{a}{t},\frac{b}{t}:0]\big)(\tau,0)=0$ for any integers $0\leq a,b<t$ with $(a,b)\neq (0,0)$. If $\varphi\neq 0$, then the number of zeros of $\varphi(\tau,z)$ (as a function of $z\in\CC$) in the fundamental domain $\CC / (\tau\ZZ+\ZZ)$ is at least $t^2-1>2d_1$, a contradiction by \cite[Theorem 1.2]{EZ85}. Therefore, $\varphi$ is identically zero. 

Secondly, we assume that $k$ is odd. In this case, $\varphi^2\in J_{2k,2d_1}[q^{2N}]$. Since $2N>\sqrt{2k\cdot 2d_1/6}$, the above result yields $\varphi^2=0$, and thus $\varphi=0$. 

\vspace{3mm} 

We then consider the case of $n>1$, $\Gamma=\SL_2(\ZZ)$, and $\chi=1$. We prove the assertion by induction on $n$. Suppose that the assertion holds in the case of $n-1$. Assume that $\varphi\in J_{k,\oplus_{j=1}^n A_1(d_j),1}[q^N]$ has the expression
\begin{equation}\label{eq:even}
\varphi =\Delta(\tau)^N \sum_{j=0}^{d_n} f_j(\tau,z_1,\dots,z_{n-1}) \phi_{-2,1}(\tau,z_n)^j \phi_{0,1}(\tau,z_n)^{d_n-j}, 
\end{equation}
where $z_j\in A_1(d_j)\otimes\CC$ and $f_j\in J_{k+2j,\oplus_{j=1}^{n-1} A_1(d_j),1}^{\w}$ for $0\leq j \leq d_n$. Recall that 
$$
N>D_n:=C_n \cdot k^{2^{-n}}\cdot \prod_{j=1}^{n}d_j^{2^{j-1-n}}
$$
We note that $C_n^2=2C_{n-1}$ and $D_n^2=2d_nD_{n-1}$. Since $k<3\cdot 2^{3-2^n}\cdot d_1$ and $d_i\leq d_j$ for any $1\leq i<j\leq n$, we deduce
$$
D_n<C_n\big( 3\cdot 2^{3-2^n}\cdot d_1 \big)^{2^{-n}}\cdot\prod_{j=1}^n d_n^{2^{j-1-n}}=d_1^{2^{-n}}\cdot d_n^{1-2^{-n}}\leq d_n. 
$$
Therefore, we can select a sufficiently large prime $t$ that satisfies 
$$
t>\sqrt{2d_n+1}, \quad N>D_n\sqrt{1+t^{-1}}, \quad \sqrt{1+t^{-1}}<\frac{d_n}{D_n}.
$$
Similarly to the case of $n=1$, we have 
\begin{align*}
&\sum_{c=0}^{t-1}\max\left\{N-\frac{d_nc(t-c)}{t^2},0\right\} \\
>&\sum_{c=0}^{t-1}\max\left\{D_n\sqrt{1+t^{-1}}-\frac{d_nct}{t^2},0\right\}\\ 
= &\frac{d_n}{t}\sum_{c=0}^{t-1}\max\left\{ \frac{tD_n}{d_n}\sqrt{1+t^{-1}}-c,0\right\}\\ 
\geq & \frac{d_n}{t}\times \left( \frac{tD_n}{d_n}\sqrt{1+t^{-1}}\right)^2\times \frac{1}{2} \\ 
=&\frac{(t+1)D_n^2}{2d_n} =\frac{t\psi_2(t)}{\psi_1(t)}D_{n-1}. 
\end{align*}
It follows that
$$ 
\sum_{c=0}^{t-1}t\psi_1\big(\mathrm{gcd}(t,c)\big)\max\left\{N-\frac{d_nc(t-c)}{t^2},0\right\} > t^2\psi_2(t)D_{n-1}. 
$$
Let $\hat{h}=(\frac{a}{t},\frac{b}{t})\in \big(A_1(d_n)\otimes\QQ\big)^2$, where $0\leq a,b<t$ are integers and $(a,b)\neq (0,0)$. 
From Lemma \ref{lem:induction} and the inductive hypothesis, we conclude that
$\big(\varphi\vert_{k,1}[\hat{h}:0]\big)(\tau,z_1,\dots,z_{n-1},0)=0$. If $\varphi\neq 0$, then the number of zeros of $\varphi(\tau,z_1,\dots,z_n)$ (here, we fix $\tau, z_1,\dots,z_{n-1}$ and view $\varphi$ as a function of $z_n$) is at least $t^2-1>2d_n$, leading to a contradiction by an analogue of \cite[Theorem 1.2]{EZ85}.  Therefore, $\varphi$ is identically zero. 

For general $\varphi$, it can be expressed as $\varphi = \phi + \psi$,
where 
\begin{align*}
\phi(\tau,z_1,\dots,z_n) = \frac{\varphi(\tau,z_1,\dots,z_n) + \varphi(\tau,z_1,\dots,z_{n-1},-z_n)}{2} \in J_{k,\oplus_{j=1}^n A_1(d_j),1}[q^N],\\
\psi(\tau,z_1,\dots,z_n) = \frac{\varphi(\tau,z_1,\dots,z_n) - \varphi(\tau,z_1,\dots,z_{n-1},-z_n)}{2}\in J_{k,\oplus_{j=1}^{n}A_1(d_j),1}[q^N].
\end{align*}
By Proposition \ref{Prop:weak} and \eqref{eq:tensor}, both $\phi$ and $\psi^2$ have the expression of type \eqref{eq:even}. Then the above result implies $\phi=0$ and $\psi^2=0$. Therefore, $\varphi$ is identically zero. 

\vspace{3mm}
  
We finally consider the general case. Let $\varphi\in J_{k,\oplus_{j=1}^n A_j(d_j),1}(\Gamma, \chi)[q^N]$, and let $d$ denote the order of $\chi$. Then we have  
$$
\Phi:=\prod_{\gamma\in \Gamma\backslash \SL_2(\ZZ)} \big(\varphi\vert_{k,1}\gamma\big)^d \in J_{dRk,\, \oplus_{j=1}^n A_1(dRd_j),\, 1}[q^{dN}],
$$
where $R=[\SL_2(\ZZ):\Gamma]$. By the specific case of the theorem that we have proved, $\Phi$ is identically zero, and therefore $\varphi=0$. We then finish the proof.
\end{proof}

Taking $n=1$ in Theorem \ref{th:k<m}, if 
$$
m>C_1d_1^\frac{1}{2}k^\frac{1}{2}m^\frac{1}{2}, \quad \text{that is}, \quad m>\frac{1}{6}d_1k,
$$
then $J_{k,d_1m}[q^m]=0$. This recovers \cite[Proposition 4.5]{AIP24}.

\section{Slope bound and symmetric formal Fourier--Jacobi series}\label{sec:sFJ}
In this section, we present two applications of our previous results. We establish a lower bound on the slope of orthogonal modular forms and describe the structure of the space of symmetric formal Fourier--Jacobi series on orthogonal groups. 

Let $M$ be an even lattice of signature $(\mathfrak{m},2)$ with dual lattice $M'$. Assume that $\mathfrak{m}\geq 3$. We label $\cA(M)$ as one of the two connected components of 
$$
\{\mathcal{Z} \in  M\otimes \CC:  (\mathcal{Z}, \mathcal{Z})=0, (\mathcal{Z},\bar{\mathcal{Z}}) < 0\}.
$$
The symmetric domain of type IV attached to $M$ is
$$
\cD(M):=\cA(M) / \mathbb{C}^{\times}=\{[\mathcal{Z}] \in  \PP(M\otimes \CC):  \mathcal{Z} \in \cA(M) \}.
$$
Let $\Orth^+(M)$ be the subgroup of $\Orth(M\otimes\RR)$ that preserves $M$ and $\cA(M)$.  Let $\mathbf{\Gamma}$ be a finite-index subgroup of $\Orth^+(M)$. The \textit{discriminant kernel}
$$
\widetilde{\Orth}^+(M)=\{ g \in \Orth^+(M):\; g(v) - v \in M, \; \text{for all $v\in M'$} \}
$$
is the most important example of such $\mathbf{\Gamma}$. 

\begin{definition}
Let $k\in \ZZ$ and $\chi: \mathbf{\Gamma}\to \CC^\times$ be a character. A holomorphic function $F: \cA(M)\to \CC$ is called a modular form of weight $k$ and character $\chi$ on $\mathbf{\Gamma}$ if it satisfies
\begin{align*}
F(t\mathcal{Z})&=t^{-k}F(\mathcal{Z}), \quad\  \text{for all $t \in \CC^\times$},\\
F(g\mathcal{Z})&=\chi(g)F(\mathcal{Z}), \quad \text{ for all $g\in \mathbf{\Gamma}$}.
\end{align*}
We denote the $\CC$-vector space of such forms by $\mathbf{M}_k(\mathbf{\Gamma},\chi)$. 
\end{definition}

We now focus on the specific case $M=2U\oplus L$, where $L$ is an even positive definite lattice, and $U$ is the unique even unimodular lattice of signature $(1,1)$. Around the one-dimensional cusp associated with $2U$, the symmetric domain $\cD(M)$ can be realized as the tube domain
$$
\cH(L):= \{Z=(\tau,\mathfrak{z},\omega)\in \HH\times (L\otimes\CC)\times \HH: 
(\im Z,\im Z)<0\}, 
$$
where $(\im Z,\im Z)=-2\im \tau \im \omega +
\latt{\im \mathfrak{z},\im \mathfrak{z}}$. 
Let $F$ be a modular form of weight $k$ and trivial character for $\widetilde{\Orth}^+(M)$. We expand $F$ on $\cH(L)$ as follows
\begin{equation}\label{eq:FJ}
F(\tau,\mathfrak{z},\omega)=\sum_{\substack{n,m\in \NN, \ell\in L' \\2nm-\latt{\ell,\ell}\geq 0}}f(n,\ell,m)q^n\zeta^\ell\xi^m=\sum_{m=0}^{\infty}\phi_m(\tau,\mathfrak{z})\xi^m,
\end{equation}
where $q=e^{2\pi i\tau}$, $\zeta^\ell=e^{2\pi i \latt{\ell, \mathfrak{z}}}$, $\xi=e^{2\pi i \omega}$. 
Then $\phi_m\in J_{k,L,m}$ for any $m\in\NN$, that is, $\phi_m$ is a holomorphic Jacobi form of weight $k$, index $m$ and trivial character on $\SL_2(\ZZ)$ for $L$. 
Note that
$$
\widetilde{\Orth}(L):=\{ g \in \Orth(L) : g(v)-v\in L, \; \text{for all $v\in L'$} \}
$$
is a finite subgroup of $\widetilde{\Orth}^+(M)$, and thus every $\phi_m$ is invariant under the action of $\widetilde{\Orth}(L)$, that is, 
$$
\phi_m(\tau,g(\mathfrak{z})) = \phi_m(\tau, \mathfrak{z}), \quad \text{for all $g\in \widetilde{\Orth}(L)$}. 
$$
We denote by $J_{k,L,m}^{\widetilde{\Orth}(L)}$ the subspace of the forms in $J_{k,L,m}$ that are invariant under the action of $\widetilde{\Orth}(L)$. 
In addition, we have the symmetric relation
$$
f(n,\ell,m)=f(m,\ell,n),\quad \text{for all} \;  (n,\ell,m)\in \NN \oplus L' \oplus \NN. 
$$
This inspires us to define the symmetric formal Fourier--Jacobi series and investigate its modularity. 

\begin{definition}
Let $k$ be a positive integer. A symmetric formal Fourier--Jacobi series of weight $k$ on $\widetilde{\Orth}^+(M)$ is a formal series 
\begin{equation}\label{eq:sFJ}
\Psi(Z)=\sum_{m=0}^{\infty} \psi_m(\tau,\mathfrak{z}) \xi^m, \quad \psi_m=\sum_{n\in\NN,\, \ell\in L'} f_m(n,\ell)q^n\zeta^\ell \in J_{k,L,m}^{\widetilde{\Orth}(L)},
\end{equation}
which satisfies the symmetric condition
$$
f_m(n,\ell)=f_n(m,\ell), \quad \text{for all} \; m,n\in \NN,\, \ell \in L'. 
$$
Let $\mathbf{FM}_k(\widetilde{\Orth}^+(M))$ denote the space of all such formal series.   
\end{definition}

We see from \eqref{eq:FJ} that $\mathbf{M}_k(\widetilde{\Orth}^+(M))$ is a subspace of $\mathbf{FM}_k(\widetilde{\Orth}^+(M))$. In fact, these two $\CC$-vector spaces are conjectured to coincide. 

\begin{conjecture}\label{conj}
For any positive integer $k$, we have $\mathbf{FM}_k(\widetilde{\Orth}^+(M))=\mathbf{M}_k(\widetilde{\Orth}^+(M))$. In particular, every symmetric formal Fourier--Jacobi series converges and defines a modular form on $\widetilde{\Orth}^+(M)$. 
\end{conjecture} 

We formulate two historical remarks on this conjecture. 

\begin{remark}
The case $L=A_1$ of Conjecture \ref{conj} was proved by Aoki \cite{Aok00}. The cases of $L=A_1(d)$ for $d=2,3,4$ were proved by Ibukiyama--Poor--Yuen \cite{IPY13}. The second named author and Williams (\cite[Section 4]{WW20} and \cite[Theorem 6.8]{WW24}) proved this conjecture for $40$ specific lattices $L$ of rank from $1$ to $8$. Recently, Aoki--Ibukiyama--Poor \cite{AIP24} proved this conjecture for all lattices $L$ of rank one. 
\end{remark}

\begin{remark}
There are several analogues of Conjecture \ref{conj} in other settings. Bruinier--Raum \cite{BR15} proved the modularity of formal series related to vector-valued Siegel modular forms on $\Sp_{2n}(\ZZ)$. Recently, they \cite{BR24} further extended the result to certain subgroups of $\Sp_{2n}(\ZZ)$. Xia \cite{Xia22} proved the modularity of formal series related to Hermitian modular forms on certain unitary groups of type $U(n,n)$. Pollack \cite{Pol24} and Flores \cite{Flo24} also proved some modularity results in certain contexts. 
\end{remark}

It is clear that both 
$$
\mathbf{M}_*(\widetilde{\Orth}^+(M))=\bigoplus_{k=1}^\infty \mathbf{M}_k(\widetilde{\Orth}^+(M)) \quad \text{and} \quad \mathbf{FM}_*(\widetilde{\Orth}^+(M))=\bigoplus_{k=1}^\infty \mathbf{FM}_k(\widetilde{\Orth}^+(M))
$$
have a natural structure of graded rings. Moreover, they are all integral domains (see, e.g., \cite[Lemma 3.2, Lemma 3.3]{AIP24}). This section aims to clarify the ring structure of $\mathbf{FM}_*(\widetilde{\Orth}^+(M))$. 

We first introduce the slope of symmetric formal Fourier--Jacobi series.

\begin{definition}
Let $k$ be a positive integer and $\Psi$ be a non-zero element of $\mathbf{FM}_k(\widetilde{\Orth}^+(M))$. We define the vanishing order of $\Psi$ at the one-dimensional cusp of type $2U$ as
$$
\ord(\Psi)=\min\{ m\in\NN : \; \text{$\psi_m\neq 0$ in the Fourier--Jacobi expansion \eqref{eq:sFJ}}\}. 
$$
We define the slope of $\Psi$ by 
$$
\varrho(\Psi) = \frac{k}{\ord(\Psi)}. 
$$
The minimal slope bound for symmetric formal Fourier--Jacobi series on $\widetilde{\Orth}^+(M)$ is written as
$$
\varrho_L = \inf\{ \varrho(\Psi): \; \Psi \in \mathbf{FM}_k(\widetilde{\Orth}^+(M)), \; \Psi\neq 0,\; k\geq 1 \}. 
$$
\end{definition}

We find a positive lower bound on $\varrho_L$ for all even positive definite lattices $L$. 

\begin{theorem}\label{th:slope}
Let $L$ be an even positive definite lattice of rank $l$ that satisfies $\bigoplus_{j=1}^l A_1(d_j)<L$, where $d_1,\dots,d_l$ are positive integers and $d_i\leq d_j$ for any $1\leq i<j\leq l$. Then we have
$$
\varrho_L \geq 3\cdot 2^{3-2^l}\cdot \prod_{j=1}^l d_j^{-2^{j-1}}. 
$$
\end{theorem}

\begin{proof}
Let $\Psi$ be a non-zero element of $\mathbf{FM}_k(\widetilde{\Orth}^+(M))$ that has an expansion of type \eqref{eq:sFJ}. By definition, $\psi_m=0$ for $m<\ord(\Psi)$. When $m=\ord(\Psi)$, $\psi_m$ is not identically zero, and we derive from the symmetric condition that its vanishing order at infinity is at least $m$ as a Jacobi form, that is, $\ord_\infty(\psi_m)\geq m$. If $k\geq 3\cdot 2^{3-2^l}\cdot d_1 m$,  then $\varrho(\Psi) = \frac{k}{m} \geq 3\cdot 2^{3-2^l}d_1$. Otherwise, we deduce from Theorem \ref{th:k<m} that
$$
m \leq 3^{-2^{-l}}\cdot 2^{1-3\cdot 2^{-l}} \cdot  \prod_{j=1}^l d_j^{2^{j-1-l}}\cdot k^{2^{-l}}\cdot m^{1-2^{-l}}, 
$$
which yields that $\varrho(\Psi)=\frac{k}{m}\geq 3\cdot 2^{3-2^l}\cdot \prod_{j=1}^l d_j^{-2^{j-1}}$. This proves the desired bound on $\varrho_L$. 
\end{proof}

We now prove the finiteness of the dimension of the $\CC$-vector space $\mathbf{FM}_k(\widetilde{\Orth}^+(M))$, and estimate the size of $\dim \mathbf{FM}_k(\widetilde{\Orth}^+(M))$ as the weight $k$ goes to infinity. 

\begin{theorem}\label{th:finite-dim}
Let $L$ be an even positive definite lattice of rank $l$. For any integer $k\geq 1$, we have
\begin{equation}\label{eq:Jacobi-estimation}
\dim \mathbf{FM}_k(\widetilde{\Orth}^+(M)) \leq \sum_{m=0}^{[k/\varrho_L]} \dim J_{k,L,m}^{\widetilde{\Orth}(L)}[q^m],    
\end{equation}
where $[x]$ denotes the integer part of $x\in \RR$ as before. In particular, $\dim \mathbf{FM}_k(\widetilde{\Orth}^+(M)) < \infty$.  Moreover, there exists a positive constant $c_L$ depending only on $L$ such that  
$$
\dim \mathbf{FM}_k(\widetilde{\Orth}^+(M)) \leq c_L \cdot k^{l+2}. 
$$ 
\end{theorem}

\begin{proof}
For any integer $m\geq 0$, we define the subspace
$$
\mathbf{FM}_k(\widetilde{\Orth}^+(M))[\xi^m]:=\{ \Psi \in \mathbf{FM}_k(\widetilde{\Orth}^+(M)) : \ord(\Psi)\geq m \}.  
$$
Then there is an exact sequence
\begin{equation}\label{eq:sequence}
0\longrightarrow \mathbf{FM}_k(\widetilde{\Orth}^+(M))[\xi^{m+1}]\longrightarrow  \mathbf{FM}_k(\widetilde{\Orth}^+(M))[\xi^m] \stackrel{P_r} \longrightarrow J_{k,L,m}^{\widetilde{\Orth}(L)}[q^m],  \end{equation}
where $P_r$ sends $\Psi$ to its leading Fourier--Jacobi coefficient $\psi_m$. By the slope bound,
$$
\mathbf{FM}_k(\widetilde{\Orth}^+(M))[\xi^{m}]=0, \quad \text{if $m> k/\varrho_L$.}
$$
Hence, \eqref{eq:sequence} and Theorem \ref{th:slope} together yield that
$$
\dim \mathbf{FM}_k(\widetilde{\Orth}^+(M)) \leq \sum_{m=0}^{[k/\varrho_L]} \dim J_{k,L,m}^{\widetilde{\Orth}(L)}[q^m]
$$
and the sum is, in fact, a finite sum. By \cite[Equation (13) and Theorem 3.1]{BS23}, there exists a constant $a_L$ depending only on $L$ such that 
$$
\dim J_{k,L,m} \leq a_L \cdot k \cdot \max\{m^l, 1\}, \quad m\geq 0.  
$$
Therefore, by Theorem \ref{th:slope}, we have
$$
\dim \mathbf{FM}_k(\widetilde{\Orth}^+(M)) \leq \sum_{m=0}^{[k/\varrho_L]}\dim J_{k,L,m} \leq a_L\cdot k \cdot \Big( 1+\sum_{m=1}^{[k/\varrho_L]} m^l\Big) \leq c_L \cdot k^{l+2},
$$
where $c_L$ is a certain positive constant depending only on $L$. This completes the proof. 
\end{proof}

We now arrive at the main theorem of this section. 

\begin{theorem}\label{th:algebraic}
As a graded $\mathbf{M}_*(\widetilde{\Orth}^+(M))$-module, $\mathbf{FM}_*(\widetilde{\Orth}^+(M))$ is of finite rank. In particular, the integral domain $\mathbf{FM}_*(\widetilde{\Orth}^+(M))$ is algebraic over its subring $\mathbf{M}_*(\widetilde{\Orth}^+(M))$. That is, for any $\Psi\in \mathbf{FM}_k(\widetilde{\Orth}^+(M))$, there exist positive integers $d$ and $k_0$ and $F_j \in \mathbf{M}_{k_0+jk}(\widetilde{\Orth}^+(M))$ for $0\leq j \leq d$ such that
$$
\sum_{j=0}^d F_j \cdot \Psi^{d-j} = 0 \quad \text{and} \quad F_0 \neq 0. 
$$
\end{theorem}

\begin{proof}
If $-id \in \widetilde{\Orth}^+(M)$, then $-id \in \widetilde{\Orth}(L)$ and $J_{k,L,m}^{\widetilde{\Orth}(L)}=0$ for any $m\geq 0$ and odd $k$, and therefore $\mathbf{FM}_{k}(\widetilde{\Orth}^+(M)) = 0$ for odd $k$. In this case, $k$ will be assumed to be even in the discussion below.  
On the one hand, by \cite[Proposition 1.2]{GHS07}, $\dim \mathbf{M}_{k}(\widetilde{\Orth}^+(M)) = O(k^{l+2})$, as $k$ goes to infinity. On the other hand, by Theorem \ref{th:finite-dim}, $\dim \mathbf{FM}_{k}(\widetilde{\Orth}^+(M)) = O(k^{l+2})$, as $k$ goes to infinity. We conclude from the two facts that the rank of the module $\mathbf{FM}_*(\widetilde{\Orth}^+(M))$ over $\mathbf{M}_*(\widetilde{\Orth}^+(M))$ is finite, that is, there are only finitely many elements of $\mathbf{FM}_*(\widetilde{\Orth}^+(M))$ that are linearly independent over $\mathbf{M}_*(\widetilde{\Orth}^+(M))$.
\end{proof}

This theorem is the basis for proving the modularity of symmetric formal Fourier-Jacobi series (see \cite[Theorem 4.5]{BR15} and \cite[Proposition 4.7]{AIP24}). We hope to use this theorem to prove Conjecture \ref{conj} in the near future. 

\begin{remark}
Let $\varrho_n$ denote the minimal slope bound for Siegel modular forms on $\Sp_{2n}(\ZZ)$. A positive lower bound on $\varrho_n$ was known 50 years ago (see \cite{Eic75} and \cite[Section 2.5]{BR15}). The exact value of $\varrho_n$ is only known for $n\leq 5$. More precisely, $\varrho_1=12$,  $\varrho_2=10$, $\varrho_3=9$, $\varrho_4=8$, $\varrho_5=54/7$. 
\end{remark}

\begin{remark}
In \cite[Section 6]{WW24}, the second named author and Williams classified $40$ root lattices $L$ for which the inequality \eqref{eq:Jacobi-estimation} degenerates into an equality (see \cite[Table 4]{WW24}). In this way, it was shown that Conjecture \ref{conj} holds for these lattices. Let $L$ be such a root lattice. By \cite[Lemma 6.2 and Theorem 6.8]{WW24}, for every weak Jacobi form $\psi_m$ of weight $k$ and index $m$ for $L$ that is invariant under $\widetilde{\Orth}(L)$ (denote the space of such forms by $J_{k,L,m}^{\w, \widetilde{\Orth}(L)}$), there exists a modular form $F$ of weight $k+12m$ on $\widetilde{\Orth}^+(2U\oplus L)$ such that $\ord(F)=m$ and its leading Fourier--Jacobi coefficient is $\Delta^m\cdot \psi_m$.  It follows that the minimal slope bound $\varrho_L$ is given by
$$
\varrho_L = \min\Big\{ 12 - \frac{k}{m} : \; \text{$k,m>0$ such that there exists a non-zero $\phi \in J_{-k,L,m}^{\w, \widetilde{\Orth}(L)}$} \Big\}. 
$$
Recall that the bigraded ring 
$$
J_{*,L,*}^{\w,\widetilde{\Orth}(L)}=\bigoplus_{k\in\ZZ}\bigoplus_{m\in\NN} J_{k,L,m}^{\w,\widetilde{\Orth}(L)}
$$
is finitely generated over $M_*(\SL_2(\ZZ))$, and the generators are known (see, e.g., \cite[Section 3.3]{WW24}). Therefore, it is easy to calculate $\varrho_L$ for these lattices $L$. For example, for irreducible root lattices $L=A_n$ with $1\leq n \leq 7$, we have $\varrho_L = 11-n$; for irreducible root lattices $L=D_n$ with $4\leq n\leq 8$, we have $\varrho_L=12-n$; for $L=E_6$, we have $\varrho_L=7$; for $L=E_7$, we have $\varrho_L=7$; for $L=A_2\oplus E_6$, we have $\varrho_L=4$; for $L=3A_2$, we have $\varrho_L=3$. In addition, by \cite[Theorem 6.13]{WW24}, for $L=A_1(n)$ with $1\leq n\leq 4$, we have $\varrho_L=12-2n$. Remark that the case of $L=A_1(n)$ corresponds to Siegel paramodular forms of degree two and level $n$.  
\end{remark}

\begin{remark}
Assume that $M$ is not of type $2U\oplus L$. If $M$ contains an isotropic plane, then modular forms on $\widetilde{\Orth}^+(M)$ also have Fourier--Jacobi expansions, while the Fourier--Jacobi coefficients will be holomorphic Jacobi forms on congruence subgroups of $\SL_2(\ZZ)$. By Theorem \ref{th:k<m}, it is not hard to generalize Theorems \ref{th:slope}, \ref{th:finite-dim} and \ref{th:algebraic} to this case. 
\end{remark}

\bibliographystyle{plainnat}
\bibliofont
\bibliography{refs}

\end{document}